\documentclass[onecolumn]{article}

\usepackage{graphicx,epsfig}
\usepackage{times}
\usepackage{hyperref}
\usepackage{amssymb, amsthm}
\usepackage {amssymb,latexsym,epic,eepic,stmaryrd,url}
\newtheorem{thm}{Theorem}
\newtheorem{clm}{Claim}

\newtheorem{lem}{Lemma}

\newtheorem{obs}{Observation}

\begin{document}
\title{\Large \bf d-Regular Graphs of Acyclic Chromatic Index at least d+2}

\author{Manu Basavaraju\thanks{Computer Science and Automation department,
Indian Institute of Science,
Bangalore- 560012,
India.  {\tt manu@csa.iisc.ernet.in}} \and L. Sunil Chandran\thanks{ {\bf (The Corresponding
Author).} Computer Science and Automation department,
Indian Institute of Science,
Bangalore- 560012,
India.  {\tt sunil@csa.iisc.ernet.in}} \and Manoj Kummini\thanks{Department of Mathematics,
 University of Kansas,
 1460 Jayhawk Blvd., Rm Snow 405
 Lawrence, KS 66045-7523. USA.  {\tt kummini@math.ku.edu}}
}

\date{}
\pagestyle{plain}
\maketitle

\begin{abstract}

An $acyclic$ edge coloring of a graph is a proper edge coloring such that there are no bichromatic cycles. The \emph{acyclic chromatic index} of a graph is the minimum number k such that there is an acyclic edge coloring using k colors and is denoted by $a'(G)$. It was conjectured by Alon, Sudakov and Zaks
 (and earlier by Fiamcik) that $a'(G)\le \Delta+2$, where $\Delta =\Delta(G)$ denotes
 the maximum degree of the graph. 
Alon et.al also raised  the question whether the  complete graphs of even order are the only regular graphs which require $\Delta+2$ colors to be acyclically edge colored. In this paper, using a simple counting argument we observe not only that this is not true, but infact all d-regular graphs with $2n$ vertices and $d > n$, requires at least $d+2$ colors. We also show that $a'(K_{n,n}) \ge n+2$, when $n$ is odd using a more non-trivial argument(Here $K_{n,n}$ denotes the complete bipartite graph with $n$ vertices on each side). This lower bound for $K_{n,n}$ can be shown to be tight for some families of complete bipartite graphs and for small values of $n$.
We also infer that for every $d,n$ such that $d \ge 5$, $n \ge 2d + 3$ and $dn$  even, there exist
$d$-regular graphs which require at least $d+2$-colors to be acyclically edge colored. 
\end{abstract}


~~~~~~

All graphs considered in this paper are finite and simple. A proper \emph{edge coloring} of $G=(V,E)$ is a map $c: E\rightarrow C$ (where $C$ is the set of available $colors$) with $c(e) \neq c(f)$ for any adjacent edges $e$,$f$. The minimum number of colors needed to properly color the edges of $G$, is called the chromatic index of $G$ and is denoted by $\chi'(G)$. A proper edge coloring c is called acyclic if there are no bichromatic cycles in the graph. In other words an edge coloring is acyclic if the union of any two color classes induces a set of paths (i.e., linear forest) in $G$. The \emph{acyclic edge chromatic number} (also called \emph{acyclic chromatic index}), denoted by $a'(G)$, is the minimum number of colors required to acyclically edge color $G$. The concept of \emph{acyclic coloring} of a graph was introduced by Gr\"unbaum \cite{Grun}. Let $\Delta=\Delta(G)$ denote the maximum degree of a vertex in graph $G$. By Vizing's theorem, we have $\Delta \le \chi'(G) \le \Delta +1 $(see \cite{Diest} for proof). Since any acyclic edge coloring is also proper, we have $a'(G)\ge\chi'(G)\ge\Delta$.

It has been conjectured by Alon, Sudakov and Zaks \cite{ASZ} that $a'(G)\le\Delta+2$ for any $G$.
We were informed by Alon that the same conjecture was raised earlier by Fiamcik  \cite {Fiamcik}.
 Using probabilistic arguments Alon, McDiarmid and Reed \cite{AMR} proved that $a'(G)\le60\Delta$. The best known result up to now for arbitrary graph, is by Molloy and Reed  \cite{MolReed} who showed that $a'(G)\le16\Delta$.

The complete graph on n vertices is denoted by $K_n$ and the complete bipartite graph with n vertices on each side is denoted by $K_{n,n}$. We denote the sides of the bi-partition by $A$ and $B$. Thus $V(K_{n,n}) = A \cup B$.

~~~~~~
%

\noindent \textbf{Our Result: }Alon, Sudakov and Zaks \cite{ASZ} suggested a possibility that complete graphs of even order are the only regular graphs which require $\Delta+2$ colors to be acyclically edge colored. Ne\v set\v ril and Wormald \cite{NesWorm} supported the statement by showing that the acyclic edge chromatic number of a random d-regular graph is asymptotically almost surely equal to $d+1$ (when $d \ge 2$). In this paper, we show that this is not true in general. More specifically we prove the following Theorems :

\begin{thm}
\label{thm:thm1}
Let $G$ be a d-regular graph with $2n$ vertices and $d >n$, then $a'(G) \ge d+2 = \Delta(G)+2$.
\end{thm}

\begin{thm}
\label{thm:thm2}
For any $d$ and $n$ such that $dn$ is even and $d \ge 5, n \ge 2d + 3$, there exists  a connected 
$d$-regular graphs that require $d+2$ colors to be acyclically edge colored.
\end{thm}

\begin{thm}
\label{thm:thm3}
$a'(K_{n,n}) \ge n+2= \Delta+2$, when n is odd.
\end{thm}

\noindent \textbf{Remarks: }
\begin{enumerate}

\item It is interesting to compare the statement of Theorem 1 to the result of \cite {NesWorm},
namely that almost all $d$-regular graphs for a fixed $d$, require only $d+1$ colors to be
acyclically edge colored. 
 From the introduction of \cite{NesWorm}, it appears that the authors expect their result for random 
$d$-regular graphs would extend to all d-regular graphs except for $K_n$, n even. From Theorem \ref{thm:thm1} and Theorem  \ref{thm:thm2}  it is clear that this is not true: There exists a large number of 
  $d$-regular graphs which require $d+2$ colors to be acyclically adge colored, even $d$ is fixed.

\item The complete bipartite graph, $K_{n+2,n+2}$ is said to have a perfect 1-factorization if the edges of $K_{n+2,n+2}$ can be decomposed into $n+2$ disjoint perfect matchings such that the union of any two perfect matchings forms a hamiltonian cycle. It is obvious from Lemma \ref{lem:lem1} that $K_{n+2,n+2}$ does not have perfect 1-factorization when $n$ is even. When $n$ is odd, some families have been proved to have perfect 1-factorization (see \cite{BMW} for further details). It is easy to see that if $K_{n+2,n+2}$ has a perfect 1-factorization then $K_{n+2,n+1}$ and therefore $K_{n+1,n+1}$ has a acyclic edge coloring using $n+2$ colors. Therefore the statement of Theorem \ref{thm:thm3} cannot be extended to the case when $n$ is even in general.
\item Clearly if $K_{n+2,n+2}$ has a perfect 1-factorization, then $a'(K_{n,n})=n+2$. It is known that (see \cite{BMW}), if $n+2 \in \{p,2p-1,p^2\}$, where $p$ is an odd prime or when $n+2 < 50$ and odd, then $K_{n+2,n+2}$ has a perfect 1-factorization. Thus the lower bound in Theorem \ref{thm:thm3} is tight for the above mentioned values of $n+2$.
\end{enumerate}


\noindent \textbf{Proof of Theorem 1:}
\begin{proof}
Observe that two different color classes cannot have $n$ edges each, since that will lead to a bichromatic cycle. Therefore at most one color class can have $n$ edges while all other color classes can have at most $n -1$ edges. Thus the number of edges in the union of $\Delta(G)+1 = d+1$ color classes is at most $n+d(n-1) < dn$, when $d >n$ (Note that dn is the total number of edges in $G$). Thus $G$ needs at least one more color. Thus $a'(G) \ge d+2 = \Delta(G)+2$.
\end{proof}

\noindent {\bf Remark:} It is clear from the proof
 that if $ n + d(n-1) + x < dn$ then even after removing $x$ edges
 from the given graph, the resulting graph still would require $d+2$ colors to be acyclically
edge colored.  

\noindent \textbf{Proof of Theorem 2:}
\begin{proof} 

If $d$ is odd, let $G' = K_{d+1}$. Else if $d$ is even let $G'$ be the complement of a perfect
matching on $d+2$ vertices. 
Let $H$ be any $d$-regular graph on $N = n-n'$ vertices. Now remove an edge $(a,a')$ from $G'$ and an edge $(b,b')$ from $H$. Now connect $a$ to $b$ and $a'$ to $b'$ to create a $d$-regular graph $G$. 
Clearly $G$  requires $d+2$ colors to be acyclically edge colored since otherwise it would mean 
that $G'-\{(a,a')\}$ is $d+1$ colorable, a contradiction in view of the Remark  following Theorem 1,
for $d \ge 5$.
\end{proof}

Complete bipartite graphs offer a interesting case since they have $d = n$. Observe that the above counting argument fails. We deal with this case in the next section.

\subsection*{Complete Bipartite Graphs}
\begin{lem}
\label{lem:lem1}
If $n$ is even, then $K_{n,n}$ does not contain three disjoint perfect matchings $M_1$, $M_2$, $M_3$ such that $M_i \cup M_j$ forms a hamiltonian cycle for $i,j \in \{1,2,3\}$ and $i \neq j$.
\end{lem}
\begin{proof}
Observe that a perfect matching of $K_{n,n}$ corresponds to a permutation of $\{1,2,\ldots,n\}$. Let perfect matching $M_i$ corresponds to permutation $\pi_i$. Without loss of generality, we can assume that $\pi_1$ is the identity permutation by renumbering the vertices of one side of $K_{n,n}$.

Suppose $K_{n,n}$ contains three perfect matchings $M_1$, $M_2$, $M_3$ such that $M_i \cup M_j$ forms a hamiltonian cycle for $i,j \in \{1,2,3\}$ and $i \neq j$.

Now we study the permutation $\pi_i^{-1}\pi_j$. Since $M_i \cup M_j$ induces a hamiltonian cycle in $K_{n,n}$, it is easy to see that the smallest $t \ge 1$ such that $(\pi_i^{-1}\pi_j)^t(1)=1$ equals $n$. It follows that, in the cycle structure of $\pi_i^{-1}\pi_j$, there exists exactly one cycle and this cycle is of length $n$. The sign of a permutation is defined as: $sign(\pi) = (-1)^k$ , where $k$ is the number of even cycles in the cycle structure of the permutation $\pi$. Recalling that $n$ is even, we have the following claim:

\begin{clm}
\label{clm:clm1}
$sign(\pi_i^{-1}\pi_j)= -1$ for $i,j \in \{1,2,3\}$ and $i \neq j$.
\end{clm}

Now with respect to $\pi_i^{-1}\pi_j$, taking $\pi_i=\pi_1$ (the identity permutation) and $\pi_j=\pi_2$ (or $\pi_3$), we infer that $sign(\pi_2)=-1$ and $sign(\pi_3)=-1$. Now $sign(\pi_2^{-1}\pi_3) = sign(\pi_2^{-1})sign(\pi_3)$ = (-1)(-1) = 1, a contradiction in view of $Claim$ \ref{clm:clm1}.
\end{proof}

~~~~~~~

\noindent \textbf{Proof of Theorem 3:}
\begin{proof}
Since $K_{n,n}$ is a regular graph, $a'(K_{n,n}) \ge \Delta +1 = n+1$. Suppose $n+1$ colors are sufficient. This can be achieved only in the following way: One color class contains $n$ edges and the remaining color classes contain $n-1$ edges each. Let $\alpha$ be the color class that has $n$ edges. Thus color $\alpha$ is present at every vertex on each side $A$ and $B$. Any other color is missing at exactly one vertex on each side.

\begin{obs}
\label{obs:obs1}
Let $\theta \neq \alpha$ be a color class. The subgraph induced by color classes $\theta$ and $\alpha$ contains $2n-1$ edges and since there are no bichromatic cycles, the subgraph induced is a hamiltonian path. We call this an $(\alpha,\theta)$ hamiltonian path.
\end{obs}

\begin{obs}
\label{obs:obs2}
Let $\theta_1$ and $\theta_2$ be color classes with $n-1$ edges each. The subgraph induced by color classes $\theta_1$ and $\theta_2$ contains $2n-2$ edges. Since there are no bichromatic cycles, the subgraph induced consists of exactly two paths.
\end{obs}

Note that there is a unique color missing at each vertex on each side of $K_{n,n}$. Let $m(u)$ be the color missing at vertex $u$. For $a_1 \in A$ and $b_1 \in B$, let $m(a_1)=m(b_1)= \beta$. Let the color of the edge $(a_1,b_1) =\gamma$. Clearly $\gamma \neq \alpha$ since otherwise there cannot be a $(\alpha,\beta)$ hamiltonian path, a contradiction to $Observation$ \ref{obs:obs1}. For $a_2 \in A$ and $b_2 \in B$, let $m(a_2)=m(b_2)= \gamma$. Its clear that $a_1 \neq a_2$ and $b_1 \neq b_2$. Consider the subgraph induced by the colors $\beta$ and $\gamma$. In view of $Observation$ \ref{obs:obs2} it consists of exactly two paths. One of them is the single edge $(a_1,b_1)$. The other path has length $2n-3$ and has $a_2$ and $b_2$ as end points.

Now we construct a $K_{n+1,n+1}$ from the above $K_{n,n}$ by adding a new vertex, $a_{n+1}$ to side $A$ and a new vertex, $b_{n+1}$ to side $B$. Now for $u \in B$ color each edge $(a_{n+1},u)$ by the color $m(u)$ and for $v \in A$ color each edge $(b_{n+1},v)$ by the color $m(v)$. Assign the color $\alpha$ to the edge $(a_{n+1},b_{n+1})$. Clearly the coloring thus obtained is a proper coloring.

Now we know that there existed a $(\alpha,\beta)$ hamiltonian path in $K_{n,n}$ with $a_1$ and $b_1$ as end points. Recalling that $m(a_1)=m(b_1)=\beta$, we have $color(a_{n+1},b_1) = color(b_{n+1},a_1)= \beta$. It is easy to see that in $K_{n+1,n+1}$ this path along with the edges $(a_1,b_{n+1})$, $(b_{n+1},a_{n+1})$ and $(a_{n+1},b_1)$ forms a $(\alpha,\beta)$ hamiltonian cycle. In a similar way, for $(\alpha,\gamma)$ hamiltonian path that existed in $K_{n,n}$, we can see that in $K_{n+1,n+1}$, we have a corresponding $(\alpha,\gamma)$ hamiltonian cycle.

Recall that there was a $(\beta,\gamma)$ bichromatic path starting from $a_2$ and ending at $b_2$ in $K_{n,n}$. In the $K_{n+1,n+1}$ we created, we have $c(a_2,a_{n+1})=\gamma$ , $c(a_1,b_{n+1})=\beta$ , $c(a_{n+1},b_1)=\beta$ and $c(a_{n+1},b_2)=\gamma$. Thus the above $(\beta,\gamma)$ bichromatic path in $K_{n,n}$ along with the edges $(a_2,b_{n+1})$, $(b_{n+1},a_1)$, $(a_1,b_1)$, $(b_1,a_{n+1})$, $(a_{n+1},b_2)$ in that order. Thus we have 3 perfect matchings induced by the color classes $\alpha$, $\beta$ and $\gamma$ whose pairwise union gives rise to hamiltonian cycles in $K_{n+1,n+1}$, a contradiction to $Lemma$ \ref{lem:lem1} since $n+1$ is even.
\end{proof}

\end{document}